\documentclass[16pt]{amsart}
\usepackage{amsfonts,amssymb,amscd,amsmath,enumerate,verbatim,calc}
\usepackage{fancyhdr}
\usepackage{tikz}
\usetikzlibrary{shapes.geometric, arrows}
\usepackage{color,soul}
\usepackage{mathtools}
\usepackage{graphicx}
\everymath{\displaystyle}
\usepackage{xcolor}
\usepackage[colorlinks,citecolor=  blue]{hyperref}

\tikzstyle{arrow} = [thick,->,>=stealth]
\tikzstyle{process} = [rectangle, minimum width=4cm, minimum height=2cm, text centered, text width=2.5cm, draw=green, fill=green!10]
\newcommand{\IFF}{\text{if and only if}}

\newcommand{\N}{\mathbb{N} }

\newcommand{\C}{\mathbb{C} }

\theoremstyle{plain}

\newtheorem{theorem}{Theorem}[section]
\newtheorem{corollary}[theorem]{Corollary}

\newtheorem{proposition}[theorem]{Proposition}

\theoremstyle{definition}
\newtheorem{definition}[theorem]{Definition}

\newtheorem{remark}[theorem]{Remark}
\newtheorem{example}[theorem]{Example}

\begin{document}
\title{ Idempotent \& Nilpotent Operators and Matrices in Bicomplex Space}
\author{Anjali}
\email{anjalisharma773@gmail.com}
\address{Department of Applied Mathematics, Gautam Buddha University, Greater Noida, Uttar Pradesh 201312, India}

\author{Akhil Prakash}
\email{akhil.sharma140@gmail.com}
\address{Department of Mathematics, Aligarh Muslim University, Aligarh, Uttar Pradesh 202002, India}

\author{Amita}
\email{amitasharma234@gmail.com}
\address{Department of Mathematics, Indira Gandhi National Tribal University, Amarkantak, Madhya Pradesh 484886, India}

\author{Neeraj Kumar Tomar}
\email{neer8393@gmail.com}
\address{Department of Applied Mathematics, Gautam Buddha University, Greater Noida, Uttar Pradesh 201312, India}

\subjclass[IMS]{Primary 15A04, 15A30; Secondary 30G35}
\date{\today}
	
\begin{abstract}
This paper explores idempotent and nilpotent operators in bicomplex spaces, focusing on their properties and behavior. We define idempotent and nilpotent matrices in this framework and derive related results. Several theorems are presented to establish conditions for the existence and behavior of bicomplex idempotent and nilpotent operators and bicomplex idempotent matrices.
\end{abstract}
\maketitle

\section{Introduction} \label{section11}

The theory of bicomplex numbers is a central focus of contemporary mathematical research, with significant progress in recent years. Numerous authors (see \cite{alpay2014basics, luna2015bicomplex, rochon2004, srivastava2003, srivastava2008}) have advanced the field, exploring diverse perspectives to elucidate their properties and establish a consistent framework for the multivariate theory of complex numbers. Recently, researchers studying matrices and linear operators (see \cite{Amita2024rank, Amita2018, anjali2023matrix, anjali2025new, prabhat2}) over various algebraic systems have made extensive contributions to mathematics. Bicomplex numbers, introduced by Segre, extend the concept of complex numbers and form a commutative ring with zero divisors. Their properties find applications in functional analysis, quantum mechanics, and signal processing.

\section{Preliminaries and Notations} \label{section12}
This section provides an introduction to bicomplex numbers and explores their key properties. It highlights several essential findings related to bicomplex numbers.\\
\noindent {\bf Bicomplex numbers:}
Bicomplex numbers are an extension of complex numbers, defined as:
\[
\xi = {u}_{1}  + {i}_{1} {u}_{2}+ {i}_{2} {u}_{3} + {i}_{1} {i}_{2} {u}_{4},
\]
where $u_1, u_2,u_3$ and $u_4$ are real numbers with 
${i}_{1} {i}_{2} \ = \ {i}_{2} {i}_{1},\; \; {i}_{1}^2 = {i}_{2}^2  =  -1.$

The collection of all bicomplex numbers is represented by  $\C_{2}$ and is referred to as the bicomplex space. 
For simplicity, $\mathbb{C}_{1}$ stands for the set of complex numbers, and  $\mathbb{C}_{0}$ indicates the set of real numbers. The bicomplex space $\C_{2}$ can be characterized in two distinct ways:
\begin{eqnarray*}
\mathbb{C}_{2} &\coloneqq & \{{u}_{1}  + {i}_{1} {u}_{2}  + {i}_{2} {u}_{3}  + {i}_{1} {i}_{2} {u}_{4}\;:\; {u}_{1} , {u}_{2} , {u}_{3} , {u}_{4} \in\mathbb{C}_{0}\}, and \\
 \mathbb{C}_{2} &\coloneqq & \{{z}_{1} + {i}_{2} {z}_{2}\;:\;  {z}_{1}, {z}_{2} \in \mathbb{C}_{1}\}.
\end{eqnarray*}

The set $\mathbb{C}_{2}$ contains zero-divisors, which makes it an algebra over $\mathbb{C}_{1}$ rather than a field. Within $\mathbb{C}_{2}$, there are exactly four idempotent elements: $0,1,e_1,e_2$, where $e_1$ and $e_2$ are two nontrivial idempotent elements, specified as follows:
 \[ e_{1} \coloneqq \frac{(1 + i_{1} i_{2})}{2}\quad  \mbox{and} \quad e_{2} \coloneqq \frac{(1 - i_{1} i_{2})}{2}.\]
 These elements stand out due to their orthogonality $(e_{1}e_{2}  = e_{2}e_{1} = 0)$ \; and the fact that they add up to 1\;$(e_{1} + e_{2} =1)$.
 \begin{eqnarray}\label{eq0} 
   \mbox{Also,} \quad  e_1^n= e_1 \quad \mbox{and}\quad e_2^n= e_2; n \in \N.
 \end{eqnarray}

\noindent{\bf Idempotent Representation and Equality Condition of Bicomplex Numbers:}
Every bicomplex number $\xi \in \C_2$ has a unique idempotent representation as a complex combination of $e_1$ and $e_2$ as follows:
 \[
\xi =  (z_{1} -  i_{1} z_{2}) e_{1} + (z_{1}  +  i_{1} z_{2}) e_{2},
\]                    
The complex numbers $(z_{1} -  i_{1} z_{2})$ and $(z_{1} +  i_{1} z_{2})$ are called the idempotent component of $\xi$ and are denoted by $\xi^-$ and $\xi^+$, respectively (cf. Srivastava [11]). Thus, the bicomplex number can be written as:
$\xi = \xi^- e_{1} + \xi^+ e_{2}$, where $\xi^-  = z_{1} - i_{1} z_{2}$ and $\xi^+ = z_{1}  + i_{1} z_{2}$. 

Furthermore, for two bicomplex numbers $\xi,\eta \in \C_2$,
\[ \xi =\eta \Leftrightarrow  \xi^-= \eta^-, \xi^+= \eta^+\]
That is, the bicomplex numbers are equal if and only if their corresponding idempotent components coincide.

 \begin{definition} (\cite{anjali2025new}, [Definition 1.4]). \label{matrix def}:  
 	A bicomplex matrix of order $m\times n$ is written as \;$ A=[\xi_{i j}]_{m \times n}$, $\xi_{i j} \in \C_2$   with each element $\xi_{i j} \in \C_2$. The collection of all such bicomplex matrices is denoted $\C_2^{m \times n}$,defined as:
\begin{eqnarray}
\C_2^{m \times n}  \eqqcolon  \Big\{[\xi_{i j}] : \; \xi_{i j} \in \C_2 ;\; i = 1,2,\ldots,m,\; j = 1,2,\ldots,n \Big\}.
\end{eqnarray} 

With usual mtrix addition and scalar multiplication, the set $\C_2^{m \times n}$ forms a vector space over the field $\C_1$. The dimension of $\C_2^{m \times n}$ over $\C_1$ is immediately given by
\begin{eqnarray}
\dim(\C_2^{m \times n})(\C_1) = 2mn.
\end{eqnarray}
Furthermore, each bicomplex matrices $A$ uniquely decomposes as
$A=\left[\xi_{i j}\right]_{m \times n}\in \C_2^{m \times n}$ can be decomposed uniquely as 
\begin{eqnarray}
 A = e_1 \ A^- \ + \ e_2 \  A^+,
\end{eqnarray}
where  $A^- =\left[\xi^-_{i j}\right]_{m \times n},\; A^+ =\left[\xi^+_{i j}\right]_{m \times n}$ are complex matrices. 
\begin{remark} \label{equality matrix}
Analogous to the concept of equality of two bicomplex numbers, two bicomplex matrices $A = e_1 A^-  +  e_2 A^+,B = e_1 B^-  +  e_2 B^+ \in\C_2^{m \times n}$  are equal if and only if their idempotent component matrices are equal. That is,
\begin{eqnarray} \label{equal}
    A=B \quad \mbox{if and only if} \quad A^- = B^- \quad \mbox{and} \quad A^+ = B^+,
    \end{eqnarray}
  and the product , sum of two bicomplex matrices and bicomplex scalar product are decomposed as follows:
  \begin{eqnarray}
      A\cdot B &=& e_1(A^- \cdot B^-) + e_2(A^+ \cdot B^+)\\
      (A+B) &=& e_1(A^- + B^-) + e_2(A^+ + B^+)\\
      \xi \cdot A &=& e_1 (\xi A^-) +  e_2 (\xi A^+) ; \ \forall \ \xi \in \C_2
  \end{eqnarray}
\end{remark}

\begin{remark} (\cite{anjali2023matrix},[Remark 3.1]).\label{rem1}
To streamline notation, denote the set of all \index{$\C_1$-linear maps from $\C_1^{n}$ to $\C_1^{m}$}$\C_1$-linear maps from $\C_1^{n}$ to $\C_1^{m}$ by $L_1^{nm}$,  and set  of all \index{$\C_1$-linear maps from $\C_2^{n}$ to $\C_2^{m}$} $\C_1$-linear maps from $\C_2^{n}$ to $\C_2^{m}$ by $L_2^{nm}$. Both are vector spaces over $\C_1$, with dimensions:
\begin{eqnarray}\label{4mn}
\dim(L_1^{nm}) = m n \;\;\mbox{and}\;\; \dim(L_2^{nm}) = \dim C_2^{n}\cdot \dim C_2^{m} = 2n\cdot 2m = 4mn. 
\end{eqnarray}
Since $\C_1$ is a field, $L_1^{nm}  \cong \C_1^{m \times n}$. However $\C_2$ is a not field, $L_2^{nm}  \not\cong  \C_2^{m \times n}$. Instead, $\C_2^{m \times n}$ is a proper subspace of $L_2^{nm}$, leading to the next definition.
\end{remark}

\begin{definition}\label{defid}
( \cite{anjali2023matrix},[Definitions 3.2, 4.1]). For any given $T_{1}, T_{2} \in L_1^{nm}$,  we can define a map $T \colon \C_2^{n}  \to \C_2^{m} $  by the following rule:
\[
T(\xi_{1}, \xi_{2},\ldots,\xi_{n}) \eqqcolon e_{1}\cdot {T}_{1} (\xi_{1}^-, \xi_{2}^-,\ldots,\xi_{n}^-)
+ e_{2}\cdot {T}_{2} (\xi_{1}^+, \xi_{2}^+,\ldots,\xi_{n}^+).
\]
\end{definition}
\noindent Clearly \;$T$\; is a $\C_1$-linear map.\;$T$\; can also be represented by $e_{1} T_{1} + e_{2} T_{2}$. Thus the set of all such linear maps is the idempotent product $L_1^{nm} \times_{e} L_1^{nm}$, i.e., we have
\begin{eqnarray}\label{idemproduct}
L_1^{nm} \times_{e} L_1^{nm} &\eqqcolon & \left \{\;e_{1} T_{1} + e_{2} T_{2} \in L_2^{nm}:\; \;T_{1}, T_{2} \in L_1^{nm}\right \}.  
\end{eqnarray}

\noindent For convenience, the set of all such type of $T= e_{1} T_{1} + e_{2} T_{2} \colon \C_2^{n}  \to \C_2^{n}$ linear operators is denoted by  $L_{1}^{n} \times_{e} L_{1}^{n}$. The idempotent product $L_1^{nm} \times_{e} L_1^{nm}$ is a subspace of $L_2^{nm}$ over the field $\C_1$. This indicates directly that $L_{1}^{nm} \times_{e} L_{1}^{nm}$ has dimension $2mn$. That is 
\begin{eqnarray}
    \dim (L_1^{nm} \times_{e} L_1^{nm}(\C_1))= 2mn.
\end{eqnarray}

Since $\C_2^{m \times n}$ and $L_{1}^{nm} \times_{e} L_{1}^{nm}$  have same dimensions over $\C_1$, they are isomorphic. Hence, the matrix expression for
$T= e_1 T_1 + e_2 T_2$ is defined using the ordered bases $\mathcal{B}_{1}$  for $\C_1^{n}$, and $\mathcal B_{2}$  for $\C_1^{m}$ as follows:

\begin{eqnarray}
[T]^{\mathcal{B}_{1}}_{\mathcal{B}_{2}}\eqqcolon e_1 [T_1]^{\mathcal{B}_{1}}_{\mathcal{B}_{2}} + e_{2} [T_2]^{\mathcal{B}_{1}}_{\mathcal{B}_{2}}.
\end{eqnarray}

Here, $[T_1]^{\mathcal{B}_{1}}_{\mathcal{B}_{2}}$ and $[T_2]^{\mathcal{B}_{1}}_{\mathcal{B}_{2}}$ are matrices of $T_1$ and $T_2$ for bases $\mathcal{B}_{1}$ and $\mathcal{B}_{2}$. If $\C_1^{n} = \C_1^{m}$, the matrix representation of $T= e_{1} T_{1} + e_{2} T_{2}$
with respect to basis $\mathcal{B}$ for $\C_1^{n}$ is simplified to  $[T]_{\mathcal{B}}$ from $[T]^{\mathcal{B}}_{\mathcal{B}}$. Thus, it follows:
\begin{eqnarray}\label{matrixrepre}
[T]_{\mathcal{B}} = e_{1} [T_{1}]_{\mathcal{B}} + e_{2} [T_{2}]_{\mathcal{B}}.
\end{eqnarray}

\end{definition}
\begin{proposition} (\cite{anjali2023matrix},[Proposition 3.3]). \label{Proposition1}
Let $T, S \in L_{1}^{nm} \times_{e} L_{1}^{nm}$ be any elements such that $T = e_{1} T_{1} + e_{2} T_{2}$ and $S = e_{1} S_{1} + e_{2} S_{2}$. Then, we have 
\begin{enumerate}
\item $T + S = e_{1} (T_{1} + S_{1}) + e_{2} (T_{2} + S_{2})$.
\item $\alpha T = e_{1} (\alpha T_{1}) + e_{2} (\alpha T_{2}); \quad \forall  \ \alpha \in \C_1$.
\end{enumerate}
\end{proposition}

\begin{theorem} (\cite{anjali2025new},[Theorem 2.7]) \label{th1}
A linear operator $T = e_1T_1 + e_2T_2 \in L_1^{n} \times_e L_1^{n}$ is singular $\IFF$  either $T_1$ is singular or $T_2$ is singular.
\end{theorem}

Previously, \cite{anjali2023matrix} introduced the
"Idempotent method" for matrix representation a linear map of the form  $T= e_1 T_1 + e_2 T_2: \C_2^{n} \rightarrow \C_2^{m} $. This method provides a systematic approach to establishing a one-to-one correspondence between bicomplex matrices $A=[\xi_{i j}]_{n \times n}$  and the linear operator's $T= e_1 T_1 + e_2 T_2 $ on finite dimensional vector space $\C_2^{n}$. This method helps analyze specific classes of matrices and operators in bicomplex spaces, offering a valuable approach for further study. For a detailed discussion on the Idempotent Method, see \cite{anjali2023matrix}.
With this foundation in place, we examine idempotent and nilpotent operators and idempotent and nilpotent matrices in bicomplex spaces,which offer unique insights into the structure of bicomplex linear algebra
\begin{theorem} (\cite{anjali2023matrix},[Theorem 3.4]). \label{properties} 
Let $T = e_{1} T_{1} + e_{2} T_{2},\; S = e_{1} S_{1} + e_{2} S_{2}$ be any two elements of $L_{1}^{nm} \times_{e} L_{1}^{nm}$. Then, we have
\begin{enumerate}
\item $T = 0$ \IFF \; $T_{1} = 0, T_{2} = 0$
\item $T = S$ \IFF  \; $T_{1} = S_{1}, T_{2} = S_{2}$
\item $S\circ T = e_{1} (S_{1} \circ T_{1}) + e_{2} (S_{2} \circ T_{2})$, wherever composition defined.
 \end{enumerate}
 \end{theorem}

Anjali \cite{anjali2023matrix}, stated Theorem \ref{properties} and we build upon this by extending the concept to the case where $T^n=0 \; \ \forall \quad n \ \in \ \N$; accordingly, we propose the following theorems.
\begin{theorem} \label{t^k}
   Let $T = e_{1} T_{1} + e_{2} T_{2}$ be a elements of $L_{1}^{nm} \times_{e} L_{1}^{nm}$. Then,  
   \begin{eqnarray*}
   T^n &=&e_1 \underbrace{(T_1 \circ T_1 \circ T_1 \ldots T_1)}_{n \text{ times}}  + e_2 \underbrace{(T_2 \circ T_2 \circ T_2 \ldots T_2)}_{n \text{ times}}\\
   \quad \mbox{Or} \quad     T^n &=& e_1 T_{1}^n + e_2 T_{2}^n ; \ \forall \  \ n \in \N
   \end{eqnarray*}
\end{theorem}
\begin{proof}
   To prove that  $T^n = e_1 T_{1}^n + e_2 T_{2}^n$, for all $n \in \N$, using the principle of mathematical induction.
   \\
    \noindent {\textbf{Case 1.}} For $n=1$, we have
    \begin{eqnarray*}
        T^1= e_1 T_1^1 +e_2 T_2^1.
    \end{eqnarray*}
    Clearly, the statement holds.\\
    Assume that the property holds for $n=k$, that is
   \begin{eqnarray*}
   T^k = e_1 T_1^k + e_2 T_2^k.
\end{eqnarray*}
We need to show that it holds for $n=k+1$, that is 
\begin{eqnarray*}
    T^{k+1} = e_1 T_1^{k+1}  + e_2 T_2^{k+2}.
\end{eqnarray*}
Since 
\begin{center}
$T^{k+1} = T^k \circ T$.    
\end{center}
We substitute $T^k$ with its assumed form:
\begin{eqnarray*}
T^{k+1} &=&[e_1 T_1^k + e_2 T_2^k] \circ [e_1 T_1 +e_2 T_2]\\
&=&e_1(T_1^k\circ T_1)+e_2(T_2^k\circ T_2)\quad\{\mbox{by Theorem \ref{properties}}\}\\
&=&[e_1 \underbrace{(T_1 \circ T_1 \circ T_1 \ldots T_1)}_{k+1 \text{ times}}  + e_2 \underbrace{(T_2 \circ T_2 \circ T_2 \ldots T_2)}_{k+1 \text{ times}}]\\
&=&e_1T_1^{k+1}+e_2T_1^{k+1}.
\end{eqnarray*}
Using the principle of mathematical induction, the result holds for every natural number $n$, i.e.
\begin{eqnarray*}
    T^n = e_1 T_{1}^n + e_2 T_{2}^n
\end{eqnarray*}
This proof holds for any linear operator $T \in L_{1}^{n} \times_{e} L_{1}^{n}$. \\
Thus the theorem is proved.
\end{proof}

\begin{theorem} \label{Nillpotent}
    Let $T = e_{1} T_{1} + e_{2} T_{2},\; S = e_{1} S_{1} + e_{2} S_{2}$ be any two elements of $L_{1}^{nm} \times_{e} L_{1}^{nm}$. Then, we have
\begin{enumerate}
\item $T^k = 0$ \IFF \; $T_{1}^k = 0, T_{2}^k = 0$
\item $T^k = S^k$ \IFF  \; $T_{1}^k = S_{1}^k, T_{2}^k = S_{2}^k$
\end{enumerate}
\end{theorem}

\begin{proof}
  (1)  We need to prove that for any element $T=e_1 T_1 + e_2 T_2 \in L_{1}^{nm} \times_{e} L_{1}^{nm}$, $T^k = 0$ \IFF \; $T_{1}^k = 0, T_{2}^k = 0$\\
    Suppose,
    \begin{eqnarray*}
        &&T^k = 0 \\
        &\Leftrightarrow& e_1 T_1^k + e_2 T_2^k =0 \quad \quad\{\mbox{by Theorem \ref{t^k}}\}\\
         &\Leftrightarrow& T_1^k  = 0 \quad \&  \quad  T_2^k =0 \quad \quad \{\mbox{ as $T^k$ is L.T. \& by part (1) of Theorem \ref{properties}}\}.
    \end{eqnarray*}
    (2) We need to prove that for any two elements $T = e_{1} T_{1} + e_{2} T_{2},\; S = e_{1} S_{1} + e_{2} S_{2}\in L_{1}^{nm} \times_{e} L_{1}^{nm}$, the equality $T^k=S^k \Longleftrightarrow  T_{1}^k = S_{1}^k, T_{2}^k = S_{2}^k$ for some $k \in \N$.\\
Suppose, 
    \begin{eqnarray*}
       && T^k = S^k, \quad\mbox{for some}\ k \in \N\\
        &\Leftrightarrow&  e_1 T_{1}^k +e_2 T_{2}^k = e_1 S_{1}^k + e_2 S_{2}^k,\quad\mbox{for some}\ k \in \N \quad \quad\{\mbox{by Theorem \ref{t^k}}\}\\
     &\Leftrightarrow& \ T_{1}^k = S_{1}^k \ \mbox{and} \ T_{2}^k = S_{2}^k, \ \mbox{for some}\ k \in \N \ \{\mbox{ as $T^k , S^k$ are L.T. \& by part (2) of Theorem \ref{properties}}\}.
    \end{eqnarray*}
    Thus the theorem is proved.
\end{proof}

\section{Bicomplex nilpotent operator and  nilpotent matrices} \label{section13}
In this section, we define bicomplex nilpotent operators and explore related results. For convenience, we introduce the terms $\C_2$-nilpotent operators and $\C_2$-nilpotent matrices to specifically refer to nilpotent operators and matrices in bicomplex spaces, respectively.

\begin{definition}{\textbf{$\C_2$-nilpotent operator:}} 
A linear operator  $T \in L_1^{n} \times_{e} L_1^{n}$  is said to be a $\C_2 $-nilpotent operator if $T^n=0$ for some positive integer $n$. The smallest such $n$ is called the index of $T$.
\end{definition}

\begin{definition} {\textbf{$\C_2 $-nilpotent matrix:}} \label{nilpotent matrix}
A matrix $A= e_1 A^-  + e_2 A^+  \in \C_2^{m \times n}$  is said to be a $\C_2 $ - nilpotent matrix if there exists a positive integer $n$ such that $A^n=0$. The smallest such $n$ is called the index of matrix $A$. 
\end{definition}

\begin{theorem}\label{t7}
A linear operator $T=e_1 T_1 + e_2 T_2\in L_1^{n} \times_{e} L_1^{n}$  is a $\C_2 $-nilpotent operator if and only if $T_1$ and $T_2$ are nilpotent operators.
\end{theorem} 
\begin{proof}
   Suppose $T$ is a nilpotent operator. Then, there exists a natural number $k$ such that
   \begin{eqnarray*}
       && T^k=0 \ \mbox{or} \ (e_1 T_1 + e_2 T_2)^k = 0\\
         &\Rightarrow&  T_1^k = 0\quad \mbox{and}\quad T_2^k = 0 \quad  \{\mbox{by Theorem} \ \ref{Nillpotent}\}\\
        &\Rightarrow&  T_1 \quad \mbox{and} \quad T_2 \quad  \mbox{will be nilpotent operators}. \quad \{\mbox{by Definition \ \ref{nilpotent matrix}}\}
   \end{eqnarray*}
Conversely:
Let $T_1, T_2 \in L_1^{nm}$ be two nilpotent operators. Then there exists natural numbers $k_1,k_2$ such that
\begin{eqnarray*}
    T_1^{k_1} = 0\quad \mbox{and}\quad T_2^{k_2}=0. 
\end{eqnarray*}
This gives that 
\begin{eqnarray} \label{eq 14}
T_1^l=0 \quad\mbox{and}\quad T_2^l=0 \quad \forall \ l\in \N; l\ge k_1,k_2.    
\end{eqnarray}

From Theorem \ref{t^k} and let $l=max(k_1,k_2)$, then
\begin{eqnarray*}
 (T)^l &=& (e_1 T_1+ e_2 T_2)^l\\
&=&e_1 (T_1)^l + e_2 (T_2)^l  \\ 
&=& e_1 0 + e_2 0 \quad\{\mbox{as} \  l \ge k_1,k_2\ \mbox{and\ by\ Equation} \ \ref{eq 14} \}\\
&=& 0.
\end{eqnarray*}
Thus, we have a natural number $l$ such that $T^l=0$. Hence, $T$ will be a nilpotent operator, as required. Thus, the proof of the theorem is complete.
\end{proof}

  \begin{theorem}\label{t'7}
    Let $T= e_1 T_1+ e_2 T_2  \in L_1^{n} \times_{e} L_1^{n}$ be a $\C_2$-nilpotent operator and 
let $\mathcal{B}_{1}$ be the ordered basis for $\C_1^{n}$ such that $[T_1]_{\mathcal{B}_{1}}=A^-$, and  $[T_2]_{\mathcal{B}_{1}}=A^+ \ \IFF$  $A=e_1 A^- + e_2 A^+ $ is $\C_2$-nilpotent matrix.
\end{theorem} 
\begin{proof}
  Suppose $T= e_1 T_1+ e_2T_2\in L_1^{n} \times_{e} L_1^{n}$ is a $\C_2$-nilpotent operator. We use  Definition \ref{defid}  and Theorem \ref{t7} ,we have
  \begin{eqnarray*}
    &&  T_1 \quad \mbox{and} \quad T_2 \quad \mbox{are nilpotent operators} \\
     &\Leftrightarrow& \exists \quad n_1,n_2\in \N\quad \mbox{such\ that}\quad T_1^{n_1}=0 \quad \mbox{and}\quad T_2^{n_2}=0\\
      &\Leftrightarrow& \exists \quad \mbox{basis} \   \mathcal{B}_1 \ \mbox{for} \ \C_1^n \ \mbox{such that} \quad  \big([T_1]_{\beta_1}\big)^{n_1} = 0 \quad \mbox{and} \quad \big([T_2]_{\beta_1}\big)^{n_2} = 0 \quad \mbox{are nilpotent  matrices}\\
     &\Leftrightarrow&e_1 \big([T_1]_{\mathcal{B}_1}\big)^n + e_2 \big([T_2]_{\mathcal{B}_1}\big)^n= e_1 (A^-)^n+ e_2 (A^+)^n =0\\
     &\Leftrightarrow& \big([e_1 T_1+ e_2 T_2 ]_{\mathcal{B}_1}\big)^n=A^n =0  \hspace{0.5cm}\{\because e_{1}^n \ = \ e_{1} ,\ \& \ e_{2}^n \ = \ e_{2}; n \in \N \}.
\end{eqnarray*}
Thus matrix $A = [e_1 T_1+ e_2 T_2 ]_{\mathcal{B}_1}$ is a $\C_2$-nilpotent matrix. Thus the proof of the theorem is complete.
\end{proof}

\begin{theorem} \label{t8}
Let $T= e_1 T_1+ e_2 T_2  \in L_1^{n} \times_{e} L_1^{n}$ be a $\C_2$-nilpotent operator. Then $T_1$ and $T_2$ are singular.
\end{theorem}
\begin{proof}
Suppose $T $  is a $\C_2$-nilpotent operator. Then, using  (\cite{beezer2015first}, Theorem 3.2.4),  (\cite{khanna2013}, Theorem 1, p.n.590), and Theorem \ref{t7} we have    \begin{eqnarray*}
    &&T_1 \quad \mbox{and} \quad T_2 \quad  \mbox{are nilpotent operator}\\
    &\Rightarrow&  \mbox{All eigenvalue of} \quad T_1 \quad \mbox{and} \quad T_2 \quad \mbox{are zero}\\
   &\Rightarrow&  (T_1-0I) \quad \mbox{is singular} \quad \mbox{and}\quad (T_2-0I) \quad \mbox{is singular}\\
   &\Rightarrow&  T_1  \quad \mbox{and} \quad T_2  \quad \mbox{are singular}.
\end{eqnarray*}
Hence $T_1$ and $T_2$ are singular,as required. Thus, the proof of the theorem is complete.
\end{proof}

The converse of Theorem \ref{t'7} is not true, as seen in the given example.
\begin{example}
Suppose $T_1(z_1,z_2,z_3)=(z_3+z_2,z_3,0)$ and $T_2(w_1,w_2,w_3)=(w_1,0,w_3)$. It is easy to see that $T_1$ and $T_2$ are singular operators. For $T_1$, we find $T_1^2= T_1(T_1(z_1,z_2,z_3))=T_1(z_3+z_2,z_3,0) = (z_3+0,0,0)$,
and  $T_1^3= T_1(T_1^2(z_1,z_2,z_3))= T_1(z_3+0,0,0) =(0,0,0)$. So,  $T_1^3=0$, the operator $T_1$ is nilpotent with index 3 because $T_1^3=0$, but $T_1^2\neq0$.
On the other hand, for all $n\ge 1$  we have $T_2^n= T_2\neq0$, which shows that $T_2$ is not nilpotent. Hence, by Theorem \ref{t7}, it follows that $ T$ is not nilpotent.
\end{example}

\begin{theorem}\label{t9}
Let $T= e_1 T_1+ e_2 T_2  \in L_1^{n} \times_{e} L_1^{n}$ be a $\C_2$-nilpotent operator and let $T_1$ and $T_2$ be two nilpotent operators of index $k_1$ and $k_2$ respectively. Then T is a $\C_2$-nilpotent operator of the index $max(k_1, k_2)$ and vice versa.
\end{theorem}

\begin{proof}
 Suppose $T_1$ and $T_2$ are nilpotent operators of index $k_1$ and $k_2$ respectively. Then
\[
T_1^{k_1}=0, \ T_1^{k_1-1} \neq0 \quad \mbox{and} \quad T_2^{k_2}=0, T_2^{k_2-1}\neq0.
\]    
\noindent {\bf Case 1:} If $k_1 \le k_2$.Then, we have 
\[
T_1^{k_2}=0
\]
Now,
\begin{eqnarray*}
T^{k_2}&=&( e_1 T_1  + e_2 T_2 )^{k_2}\\
&=& e_1T_1^{k_2}+e_2T_2^{k_2}\\
&=&0  \hspace{1.6in}\{\mbox{as} \ T_1^{k_2}=0 \ \mbox{and} \  T_1^{k_2}=0\}\\
 \mbox{and} \quad 
 T^{{k_2}-{1}} = (e_1 T_1  + e_2 T_2)^{k_2-1} &=& e_1 T_1^{k_2-1}  + e_2 T_2^{k_2-1} \neq 0 \quad \quad\{\mbox{as} \  T_2^{k_2-1}\neq 0\ \mbox{and\ by\ Theorem \ref{t^k}}\}
\end{eqnarray*} 
Therefore $T$ is a $\C_2$-nilpotent operator of index $k_2$ \\

\noindent {\bf Case 2:} If $k_2 < k_1$.Then we can easily prove that as previous $T$ is a $\C_2$-nilpotent operator of index $k_1$.
Hence $T$ will be the $\C_2$-nilpotent operator of index $max(k_1,k_2)$.\\

\noindent{\bf Conversely:} Suppose $T$ is a $\C_2$-nilpotent operator of index $k$ such
that
\[ 
T^k = (e_1 T_1+ e_2 T_2)^k=0.
\]
Now, using Theorem \ref{t7} we have $T_1$ and $T_2$ are nilpotent.
There exist natural numbers $ k_1$ and $k_2$ such that
\[ T_1^{k_1}=0,T_1^{k_1-1}\neq0 \quad \mbox{and} \quad T_2^{k_2}=0,T_2^{k_2-1} \neq0 .\]

\noindent{\bf Case1:} If $k_1\le k_2$. Then \[ (T_1)^{k_2} = 0. \]
Thus we have, \[  T^{k_2} = e_1 (T_1)^{k_2} + e_2(T_2)^{k_2}=0, \]

\fontsize{10pt}{11pt}\selectfont
and \[ T^{k_2-1}= e_1 (T_1)^{k_2-1}+ e_2 (T_2)^{k_2-1} \neq0 \quad (\mbox{as} \ T_2^{k_2-1}\neq0).\]  
Therefore  $T$  is a $\C_2$ -nilpotent operator of index $k_2$.\\
\noindent{\bf Case2:} If $k_2<k_1$. Then we have \[ (T_2)^{k_1} = 0. \]
Thus \[ T^{k_1} = e_1 (T_1)^{k_1}  + e_2 (T_2)^{k_1} = 0,  \] and \[  T^{k_1-1} = e_1 (T_1)^{k_1-1}+ e_2 (T_2)^{k_1-1}\neq0  \quad (\mbox{as} \ T_1^{k_1-1}\neq0). \]
Therefore $T$ is a $\C_2$-nilpotent operator of index $k_1$.\\
Since if $k_1\le k_2$ and $k_2<k_1$, then $T$ is a $\C_2$-nilpotent operator of index $k_2$ and $k_1$ respectively. Hence $T$ is a $\C_2$-nilpotent operator of index $Max(k_1,k_2)=k$.  
\end{proof}
The following Corollary \ref{t10} is immediate consequence of Theorem \ref{t9}.
\begin{corollary} \label{t10}
Suppose $T= e_1 T_1+ e_2 T_2  \in L_1^{n} \times_{e} L_1^{n}$ is a $\C_2$-nilpotent operator of index $m$. Then at least one nilpotent operator $T_1$ and $T_2$ will be of index m. 
\end{corollary}

\section{Bicomplex Idempotent operator and Idempotent matrices} \label{section14}
In this section, we define bicomplex idempotent operators and matrices and explore their related results. For convenience, we introduce the terms $\C_2$-idempotent operators and $\C_2$-idempotent matrices to refer specifically to idempotent operators and matrices in bicomplex spaces.

\begin{definition}{\textbf{$\C_2$-idempotent operator:}} \index{bicomplex idempotent operator}
\label{d1}
A linear operator  $T= e_1 T_1 +e_2 T_2 \in L_{1}^{n} \times_{e} L_{1}^{n}$  is said to be $\C_2 $ - idempotent operator or bicomplex idempotent operator if $T^2=T$. 
\end{definition}
\begin{definition}{\textbf{$\C_2$-idempotent Matrices:}} \index{bicomplex idempotent matrix}
\label{d2}
  A Matrix $A= e_1 A^- +e_2 A^+ \in \C_2^{n\times n}$ is said to be $\C_2$-idempotent matrix or bicomplex idempotent matrix if $A^2=A$.
\end{definition}

\begin{theorem} \label{t1}
A linear operator $T \in L_1^{n} \times_{e} L_1^{n}$ is a $\C_2 $-idempotent if and only if $T_1$ and $T_2$ are the idempotent linear operator.
\end{theorem} 

 \begin{proof}
Suppose $T= e_1 T_1 +e_2 T_2 \in L_1^{n} \times_{e} L_1^{n}$ is a $\C_2 $-idempotent linear operator. We use Definition \ref{d1} and Theorem \ref{properties}  throughout this proof.
\begin{eqnarray*}
&& \hspace{1.2cm}T^2=T  \\
&\Leftrightarrow& e_1^2 T_1^2  +  e_2^2 T_2^2 + 2e_1 e_2 T_1 T_2 =  e_1 T_1+  e_2 T_2\\
&\Leftrightarrow&  e_1 T_1^2  + e_1 T_2^2  = e_1 T_1  + e_2 T_2 \quad \{\because \ e_1. e_2 = e_2.e_1 = 0 \ , e_1^2 =e_1  \ \&  \ e_2^2 = e_2\} \\
&\Leftrightarrow& \quad T_1^2 = T_1 \ \mbox{and} \ T_2^2 = T_2 \quad \{\mbox{ $\because$ $T^2$ is L.T. $\&$ by part $(2)$ of Theorem \ref{properties}}\}\\
&\Leftrightarrow& \quad  T_1 \quad \mbox{and} \quad T_2 \quad \mbox{will be idempotent operators},
\end{eqnarray*}

as required. Thus the proof of the theorem is complete.
 \end{proof}
 
The following properties of $\C_2$-idempotent operators provide fundamental insight into their structure, composition, and algebraic significance.
 
\textbf{Properties:}\label{properties ID} {Let $T=e_1T_1 + e_2 T_2 , S = e_1 S_1 + e_2 S_2$ be any two elements of $L_1^{n} \times_{e} L_1^{n}$. Suppose  $S$, $T$ are $\C_2$-idempotent operators. Then we have:
\begin{enumerate}
\item $I-T$ is a $\C_2$-idempotent \index{$I-T$ is $\C_2$-idempotent} operator \IFF \; $I-T_1, I-T_2$ are idempotent operator. Where $ I = e_1 I^- + e_2 I^+$ is the identity operator.
\item $S\circ T$ is a $\C_2$-idempotent operators \index{$S\circ T$ is $\C_2$-idempotent} \IFF  \; $ S_{1} \circ T_{1}, S_{2} \circ T_{2}$ are idempotent opeartor.
\item $S+ T$ is a $\C_2$-idempotent \index{$S+ T$ is a $\C_2$-idempotent} \IFF \quad $(S_{1} + T_{1}) , (S_{2} + T_{2})$ are idempotent operators, provided $ST=0, TS=0.$
\end{enumerate}}

 The following Theorem  \ref{t2} true for bicomplex matrix can be verified easily.
 \begin{theorem} \label{t2}
Let $A= e_1 A^- +e_2 A^+$ be a matrix in  $\C_2^{n \times n}$. Then $A$  is  a $\C_2$-idempotent matrix if and only if $A^-$ and $A^+$ are idempotent matrix.
\end{theorem}

\begin{proof}
Suppose $A= e_1 A^- +e_2 A^+$ is a $\C_2$-idempotent matrix. Then by using Definition \ref{d2}, we have
\begin{eqnarray*}
 && \hspace{1.2cm}A^2 =A\\
 &\Leftrightarrow& (e_1 A^- +e_2 A^+)^2 = A_1 e_1 +A_2 e_2 \\
 &\Leftrightarrow& (e_1^2 (A^-)^2   + e_2^2 (A^+)^2  + 2( e_1 A^-) (e_2 A^+) = e_1 A^- + e_2 A^+\\
 &\Leftrightarrow& e_1 (A^-)^2  + e_2 (A^+)^2  =  e_1 A^- +  e_2 A^+ \quad \quad \{\because e_1.e_2=0 , e_1^2 =e_1, e_2^2 =e_2\}\\
&\Leftrightarrow& (A^-)^2 = A^-  \quad \mbox{and} \quad (A^+)^2  =A^+   \quad \mbox{\{by Remark \ref{equality matrix}}\}\\
 &\Leftrightarrow& A^- \quad \mbox{and}\quad A^+ \quad \mbox{are idempotent matrices}, 
\end{eqnarray*}
as required. Thus the proof of the theorem is complete.
\end{proof}

 \begin{theorem} \label{t3}
    Let $T= e_1 T_1+ e_2T_2  \in L_1^{n} \times_{e} L_1^{n}$ be a $\C_2$-idempotent operator and 
let $\mathcal{B}_{1}$ be the ordered basis for $\C_1^{n}$ such that $[T_1]_{\mathcal{B}_{1}}=A^-$, and  $[T_2]_{\mathcal{B}_{1}}=A^+ \ \IFF$ $A= e_1 A^- + e_2 A^+$ is $\C_2$-idempotent matrix.
\end{theorem}

\begin{proof}
  Suppose $T= e_1 T_1+ e_2T_2\in L_1^{n} \times_{e} L_1^{n}$ is a $\C_2$-idempotent  operator. Use  Definitions \ref{defid}, \ref{d2} and Theorem \ref{t1},we have
  \begin{eqnarray*}
     &\Leftrightarrow&  T_1 \quad \mbox{and} \quad T_2 \quad \mbox{are idempotent operator} \\
     &\Leftrightarrow& \exists \quad \mbox{basis} \quad  \mathcal{B}_{1} \quad \mbox{such that} \quad  [T_1]_{\beta_1} = A^- \quad \mbox{and} \quad [T_2]_{\beta_1} = A^+ \quad \mbox{are idempotent matrices}\\
     &\Leftrightarrow&  e_1 [T_1]_{\beta_1} + e_2 [T_2]_{\beta_1} = e_1 A^- + e_2 A^+  \ \mbox{is $\C_2$-idempotent matrix}\quad \{ \mbox {by Theorem \ref{t2}}\}\\
     &\Leftrightarrow& A= e_1 A^- + e_2 A^+ \mbox{is a $\C_2$ - idempotent matrix},
  \end{eqnarray*}
   as required. Thus the proof of the theorem is complete.
\end{proof}

\begin{theorem} \label{t4}
 A bicomplex matrix $A= e_1 A^- + e_2 A^+ \in \C_2^{n\times n}$ is a $\C_2$-idempotent matrix if and only if $e_1 A $   is a $\C_2$ - idempotent matrix \index{$e_1 A $   is $\C_2$-idempotent matrix}.
\end{theorem}

\begin{proof}
    Suppose $A= e_1 A^- + e_2 A^+$ is a $\C_2$-idempotent matrix. Use Definition \ref{d2} and Theorem \ref{t2}, we have
    \begin{eqnarray*}
A^2 &=& A\\
\Leftrightarrow (A^-)^2 = A^-  &\mbox{and} & (A^+)^2  =A^+    \quad \quad \{\mbox{by Theorem} \ref{t2}\}\\
        \mbox{Now,} \quad (e_1 A)^2 &=& e_1^2 [e_1^2 (A^-)^2 +e_2^2 (A^+)^2 +2A^- A^+ e_1 e_2]\\
        &=& e_1 [e_1 (A^-)^2 +e_2 (A^+)^2] \quad \quad (\mbox{Since} \quad e_1.e_2=0 , e_1^2 =e_1, e_2^2 =e_2)\\
        &=&  e_1 (e_1 A^- + e_2 A^+) \\
        &=&  e_1 A.
    \end{eqnarray*}
    Hence $e_1 A$ is a $\C_2$-idempotent matrix, as required. Thus the proof of the theorem is complete.
\end{proof}
The following Corollary \ref{c1} is immediate consequence of Theorem \ref{t4}.

\begin{corollary} \label{c1} 
    A bicomplex matrix $A= e_1 A^- + e_2 A^+  \in \C_2^{n\times n}$ is a $\C_2$-idempotent matrix if and only if $e_2 A $   is a $\C_2$-idempotent matrix \index{$e_2 A $   is $\C_2$-idempotent matrix}.
\end{corollary}

\begin{theorem} \label{t5}
   Let $A= e_1 A^- + e_2 A^+, B= e_1 B^- +e_2 B^+  \in \C_2^{n\times n}$.  Then $A$ and $B$ are  $\C_2$-idempotent matrices if and only if $e_1 A + e_2 B$  is a $\C_2$ - idempotent matrix.
\end{theorem}

\begin{proof}
    Suppose $A= e_1 A^- + e_2 A^+$ and $ B= e_1 B^- +e_2 B^+$ are $\C_2$-idempotent matrix. Use Definition \ref{d2}, we have
    \begin{eqnarray*}
    && A^2=A \quad \mbox{and} \quad B^2=B\\
        \mbox{Now,} \quad (e_1 A + e_2 B)^2 &=&  (e_1^2 A^2 + e_2^2 B^2 + 2(A B) (e_1 e_2))\\
        &=& e_1 A^2 +e_2 B^2 \quad \{ \because \quad e_1^2 = e_1, e_2^2=e_2,\ \&  \ e_1 . e_2 =0\}\\
        &=& e_1 A + e_2 B.
    \end{eqnarray*}
    Hence $e_1 A + e_2 B$ is a $\C_2$-idempotent matrix \index{$e_1 A + e_2 B$ is $\C_2$-idempotent matrix}, as required. Thus the proof of the theorem is complete.
\end{proof}

\begin{theorem}\label{t6}
 Let $A= e_1 A_1 +e_2 A_2$ be a
 $\C_2$ - idempotent matrix. Then $e_1(I-A)$  is also a $\C_2$ - idempotent matrix \index{$e_1(I-A)$ is $\C_2$-idempotent matrix}, where $I = e_1 I^-+ e_2 I^+$ is identity matrix of order $n \times n$.   
\end{theorem}

\begin{proof}
    Suppose $A= e_1 A_1 +e_2 A_2$ is a $\C_2$-idempotent matrix. Use Definition \ref{d2}, we have
\begin{eqnarray*} 
 A^2=A \\
\mbox{Now}, \quad   [e_1(I-A)]^2 &=& e_1^2 (I-A)^2\\
            &=& e_1^2 ( I^2+A^2 -2AI)\\
                &=&e_1(I+A-2A)  \quad \{ \because \quad e_1^2 = e_1\}\\
                &=& e_1(I-A).
\end{eqnarray*}
Hence $e_1(I-A)$ is a $\C_2$-idempotent matrix, as required. Thus the proof of the theorem is complete.
\end{proof}
The following Corollary \ref{c2} is immediate consequence of Theorem \ref{t6}.

\begin{corollary}\label{c2}
Let $A= e_1 A^- +e_2 A^+$ be a $\C_2$-idempotent matrix. Then $e_2(I-A)$ is also a $\C_2$-idempotent matrix, where $I = e_1 I^-+ e_2 I^+$ is identity matrix of order $n \times n$.       
\end{corollary}

\section*{Conclusion}
In this paper, we explored the concepts of idempotent and nilpotent operators within the framework of bicomplex spaces. We also analyzed their fundamental properties and results. Additionally, we introduced the notion of idempotent matrices in bicomplex spaces and derived several important results related to their structure and properties.

The theorems establish a foundation for understanding bicomplex idempotent and nilpotent operators, highlighting their algebraic and analytical properties. These results extend matrix and operator theory to the bicomplex setting and provide a basis for further spectral theory and functional analysis research. The findings also pave the way for further research in spectral theory, functional analysis, and applications involving bicomplex structures.

\bibliographystyle{plain}
\bibliography{references}
\end{document}